\documentclass[12pt]{amsart}


\author{Benjamin Linowitz}
\address{University of Michigan \\ Department of Mathematics \\ 530 Church Street \\ Ann Arbor, MI 48109 \\ United States}
\email{linowitz@umich.edu}

\author{Lola Thompson}
\address{University of Georgia\\ Department of Mathematics \\ Boyd Graduate Research Center \\ Athens, GA 30601 \\ United States}
\email{lola@math.uga.edu}
 
\thanks{Benjamin Linowitz was partially supported by NSF RTG grant DMS-1045119 and Lola Thompson was partially supported by NSF VIGRE grant DMS-0738586.}

\title{The sign changes of Fourier coefficients of Eisenstein series}
\usepackage{amsfonts,amsmath,amssymb,amsthm,geometry,url,hyperref}
\usepackage{amsmath, amssymb, paralist,tabularx,supertabular,amsmath, verbatim, doc, amsthm, amssymb, mathrsfs, manfnt}
\DeclareMathAlphabet{\curly}{U}{rsfs}{m}{n}

\DeclareMathOperator{\Aut}{Aut}
\DeclareMathOperator{\Gal}{Gal}
\DeclareMathOperator{\cond}{cond}
\DeclareMathOperator{\Tr}{Tr}
\newtheorem{thm}{Theorem}[section]
\newtheorem{prop}[thm]{Proposition}

\newtheorem{corollary}[thm]{Corollary}
\newtheorem{conj}[thm]{Conjecture}

\newtheorem{lemma}[thm]{Lemma}
\theoremstyle{remark}
\theoremstyle{definition}

\newtheorem*{rmk}{Remark}

\begin{document}
\def\phi{\varphi}

\renewcommand{\labelenumi}{(\roman{enumi})}
\def\polhk#1{\setbox0=\hbox{#1}{\ooalign{\hidewidth
    \lower1.5ex\hbox{`}\hidewidth\crcr\unhbox0}}}
\newcommand{\del}{\ensuremath{\delta}}
\def\A{\curly{A}}
\def\B{\curly{B}}
\def\e{\mathrm{e}}
\def\E{\curly{E}}
\def\F{\mathbf{F}}
\def\C{\mathbf{C}}
\def\I{\curly{I}}
\def\N{\mathbf{N}}
\def\D{\curly{D}}
\def\Q{\mathbf{Q}}
\def\O{\curly{O}}
\def\V{\curly{V}}
\def\W{\curly{W}}
\def\Z{\mathbf{Z}}
\def\p{\tilde{p}}
\def\Pp{\curly{P}}
\def\pr{\mathfrak{p}}
\def\Proj{\mathbf{P}}
\def\q{\mathfrak{q}}
\def\Ss{\curly{S}}
\def\T{\curly{T}}
\def\Nm{\mathcal{N}}
\def\cont{\mathrm{cont}}
\def\ord{\mathrm{ord}}
\def\rad{\mathrm{rad}}
\def\lcm{\mathop{\mathrm{lcm}}}

\newcommand{\lp}{\ensuremath{\left(}}
\newcommand{\rp}{\ensuremath{\right)}}

\numberwithin{equation}{section}
\begin{abstract} 

In this paper we prove a number of theorems that determine the extent to which the signs of the Hecke eigenvalues of an Eisenstein newform determine the newform. We address this problem broadly and provide theorems of both individual and statistical nature. Many of these results are Eisenstein series analogues of well-known theorems for cusp forms. For instance, we determine how often the $p^{th}$ Fourier coefficients of an Eisenstein newform begin with a fixed sequence of signs $\varepsilon_p = \{\pm 1, 0\}$. Moreover, we prove the following variant of the strong multiplicity-one theorem: an Eisenstein newform is uniquely determined by the signs of its Hecke eigenvalues with respect to any set of primes with density greater than $1/2$.

\end{abstract}
\maketitle

\section{Introduction}

Many results in the theory of classical elliptic modular forms are concerned with studying the extent to which a modular form is determined by its Fourier coefficients. In this paper we focus our attention on the signs of the Fourier coefficients of Eisenstein series with real coefficients and, in particular, those that are Eisenstein newforms. This is equivalent to studying the Hecke eigenvalues of the newforms in question, as it is well-known that the $p$-th Hecke eigenvalue of an Eisenstein newform is equal to its $p$-th Fourier coefficient. These Fourier coefficients are in turn given by the values of a variant of the sum-of-divisors function, $\sigma(n):=\sum_{d\mid n} d$. The variant that we consider introduces signed terms, weighting each divisor of $n$ by the values of certain Dirichlet characters. We exploit properties of Dirichlet characters and the sum-of-divisors function in order to obtain generalizations of several well-known theorems for cusp forms. Owing to the explicit nature of Eisenstein series, many of our results are ``best possible''.  We now develop the notation necessary to describe our results in greater detail.

For a positive integer $N$, Dirichlet character $\chi$ and integer $k\geq 2$, let $M_k(N,\chi)$ denote the complex vector space of modular forms on $\Gamma_0(N)$ of weight $k$ and character $\chi$. Let $E_k(N,\chi)$ (respectively $S_k(N,\chi)$) denote the subspace of Eisenstein series (respectively cusp forms). For any prime $p$, let $T_p$ denote the $p$th Hecke operator. It is well-known that $S_k(N,\chi)$ has a basis consisting of newforms, which are simultaneous eigenforms for the algebra generated by $\{T_p : (p,N)=1\}$, and their shifts by divisors of $NM^{-1}$ \cite{atkin-lehner, li}. The strong multiplicity-one theorem shows that these newforms are uniquely determined by their eigenvalues for all but finitely many of the operators $\{ T_p : (p,N)=1\}$. Ramakrishnan \cite{Ramakrishnan-multiplictyone} obtained an even stronger multiplicity-one theorem, showing that newforms are uniquely determined by their eigenvalues with respect to the $p$th Hecke operator for any set of primes with asymptotic density greater than $7/8$.

Let $f\in S_k(N,\chi)$ be a newform, $\lambda_f(p)$ the eigenvalue of $f$ with respect to the Hecke operator $T_p$ and assume that all of the eigenvalues $\{\lambda_f(p)\}$ are real. The sequence of {\em signs} of the Hecke eigenvalues of $f$ have been studied by a number of authors \cite{kohnen2006,kohnen2007,kohnen2008,kohnen2009,kowalskietal,matomaki}. It follows from a classical theorem of Landau and an analysis of the Rankin-Selberg zeta function of $f$ that there are infinitely many primes for which $\lambda_f(p)>0$ and infinitely many for which $\lambda_f(p)<0$ \cite[page 173]{kohnen2006}. In analogy with the problem of determining the least quadratic non-residue, one may consider the problem of determining the smallest prime $p$ for which $\lambda_f(p)<0$ (or, more generally, the least integer $n$ coprime to $N$ for which the eigenvalue $\lambda_f(n)$ of $f$ with respect to the Hecke operator $T_n$ is negative). The strongest result in this direction is due to Matom{\"a}ki \cite[Theorem 1]{matomaki}, who has shown that $\lambda_f(n)<0$ for some integer $n\ll (k^2N)^{3/8}$. In a different direction, Kowalski, Lau, Soundararajan and Wu \cite[Theorem 3]{kowalskietal} considered an arbitrary sequence of signs $\{ \varepsilon_p\}$ and obtained a lower bound for the proportion of newforms of $S_k(N,\chi)$ whose eigenvalue sequence has signs coinciding with $\{ \varepsilon_p\}$ for all primes $p\leq x$.

In his thesis \cite{weisinger-thesis}, Weisinger developed a newform theory for the space $E_k(N,\chi)$.  In this paper we consider questions analogous to the ones above for newforms lying in $E_k(N,\chi)$. Let $E\in E_k(N,\chi)$ be an Eisenstein newform whose Hecke eigenvalues $\{ \lambda_E(p)\}$ are all real. We show (Corollary \ref{corollary:densityresult}) that not only is $\lambda_E(p)$ positive for infinitely many primes and negative for infinitely many primes, but in fact we prove that the asymptotic density of the set of positive integers $n$ for which $\lambda_E(n)<0$ is equal to $1/2$. Using a classical result of Burgess, we show (Theorem \ref{thm:smallestsignchange}) that for any fixed $\varepsilon>0$, $\lambda_E(p)<0$ for some prime $p\ll N^{\frac{1}{4\sqrt{e}}+\varepsilon}$. In one of our main results (Theorem \ref{mainanalytic}), we consider a fixed sequence of primes $\{ p_1,\dots,p_\ell\}$ and a fixed sequence of signs $\{ \varepsilon_{p_1},\dots,\varepsilon_{p_\ell} \}$ (where $\varepsilon_{p_i}\in\{-1,0,1\}$) and determine an asymptotic for the number of newforms $E\in E_k(N,\chi)$ with $N\leq x$ for which $\lambda_E(p_i)$ has sign $\varepsilon_{p_i}$ for $i=1,\dots,\ell$. 

In developing a newform theory for $E_k(N,\chi)$, Weisinger proved a strong multiplicity-one theorem in analogy with the classical strong multiplicity-one theorem for cuspidal newforms. This multiplicity-one theorem was later improved upon by Rajan \cite{rajan}, who proved an analogue of Ramakrishnan's refinement of the multiplicity-one theorem. Our Theorem \ref{thm:strongmultone} is a further refinement of the strong multiplicity-one theorem and shows that Eisenstein newforms are uniquely determined by their eigenvalues with respect to the $p$th Hecke operator for any set of primes with density greater than $1/2$. This theorem is in fact best possible; it is easy to exhibit, via quadratic twists, Eisenstein newforms whose Hecke eigenvalues coincide on a set of primes having density equal to $1/2$. In Theorem \ref{thm:quadtwists}, we clarify the extent to which distinct newforms whose Hecke eigenvalues coincide on a set of primes having density equal to $1/2$ must arise from such a quadratic twisting construction. In particular, we show that if $E_1,E_2$ are Eisenstein newforms whose Hecke eigenvalues differ on a set of primes having density $1/2$, then there exists a quadratic character $\theta$ such that for all primes $p$ with $(p,N)=1$, the $p$th Hecke eigenvalues of $E_1$ and $E_2$ differ by $\theta(p)$. We additionally show that a stronger mulitplicity-one result is true: Eisenstein newforms are uniquely determined by the signs of their Hecke eigenvalues with respect to any set of primes with density greater than $1/2$. Here we adopt the convention that the sign of a complex number $z$ is equal to $\frac{z}{|z|}$. This complements a result of Matom{\"a}ki \cite[Theorem 2]{matomaki}, which shows that a cuspidal newform $f$ of trivial character and without complex multiplication is determined by the sign of $\lambda_f(p)$ for any set of primes with analytic density greater than $19/25$.

The final section of this paper considers the problem of studying the signs of the Fourier coefficients of Eisenstein series more broadly. We begin by determining the possible newform decompositions of an Eisenstein series whose Fourier coefficients are all rational numbers (Theorem \ref{thm:rationalcoefficients}). We then show that every Eisenstein series with rational coefficients whose newform decomposition does not include the unique Eisenstein series newform whose Fourier coefficients are all positive must have negative Fourier coefficients of arbitrarily large absolute value (Theorem \ref{thm:nonnegative}). Intuitively, this result shows that one cannot take Eisenstein series newforms having negative Fourier coefficients and cleverly add them together so as to obtain a modular form whose Fourier coefficients are all non-negative. As an immediate application of these results, consider a positive definite integral quadratic form $Q$ in an even number of variables. The theta series associated to $Q$ can be decomposed as the sum of a cusp form and an Eisenstein series and has Fourier coefficients $r_Q(n)$ equal to the number of ways that $n$ is represented by $Q$. These representation numbers are always non-negative. As it is well-known that standard eigenvalue estimates imply that the Fourier coefficients of an Eisenstein series will eventually dominate those of a cusp form, Theorem \ref{thm:nonnegative} provides insight into the possible newform decompositions of the ``Eisenstein part'' of the theta series of $Q$.

\section{Preliminaries}\label{section:prelim}

Let $N_1,N_2$ be positive integers and $\chi_1,\chi_2$ be Dirichlet characters modulo $N_1,N_2$. For a positive integer $k\geq 2$, we define the following variant of the sum-of-divisors function: \begin{equation}\label{equation:sigmadef}\sigma_{\chi_1,\chi_2}^{k-1}(n)=\sum_{d\mid n}\chi_1(n/d)\chi_2(d)d^{k-1}.\end{equation}

Associated to the triple $(\chi_1,\chi_2,k)$ is a function

\begin{equation}
E(\chi_1,\chi_2,k)(z)=\frac{\delta(\chi_1)}{2}L(1-k,\chi_2)+\sum_{n\geq 1}\sigma_{\chi_1, \chi_2}^{k-1}(n)q^n,\qquad q=e^{2\pi i z},
\end{equation}

\noindent where $L(s,\chi_2)$ is the Dirichlet L-function of $\chi_2$ and $\delta(\chi_1)=1$ if $\chi_1$ is principal and is equal to $0$ otherwise. Assume that we are not in the situation that $\chi_1,\chi_2$ are both principal characters modulo $1$ and $k=2$. Then it is well-known that if $\chi_1\chi_2(-1)=(-1)^k$ then $E(\chi_1,\chi_2,k)$ is an Eisenstein series lying in $E_k(N_1N_2,\chi_1\chi_2)$.

In his thesis, Weisinger \cite{weisinger-thesis} developed a newform theory for the space $E_k(N,\chi)$ of Eisenstein series. This theory was analogous to the one developed by Atkin and Lehner \cite{atkin-lehner} for cusp forms and which was later extended by Li \cite{li}. In this theory, the newforms of $E_k(N,\chi)$ are functions $E(\chi_1,\chi_2,k)$ for which $N=N_1N_2$, $\chi=\chi_1\chi_2$ and $\chi_1,\chi_2$ are both primitive. In particular, Weisinger showed that $E_k(N,\chi)$ has a basis consisting of Eisenstein newforms of level $M\mid N$ and their shifts by divisors of $NM^{-1}$. It is easy to check that if $E(\chi_1,\chi_2,k)$ is an Eisenstein newform then it is an eigenform for all of the Hecke operators $T_p$. 
If $p$ is a prime not dividing $N$ then, by explicitly computing the action of the Hecke operator $T_p$ on $E(\chi_1,\chi_2,k)$, one sees that the eigenvalue of $E(\chi_1,\chi_2,k)$ with respect to $T_p$ is equal to $\sigma_{\chi_1,\chi_2}^{k-1}(p)$. Weisinger additionally showed that an Eisenstein newform is uniquely determined by its Hecke eigenvalues in the sense that two newforms whose eigenvalue with respect to the Hecke operator $T_p$ agree for all but finitely many primes $p$ must, in fact, be equal.

As we are interested in studying the signs of the Fourier coefficients of Eisenstein newforms, we note that an immediate consequence of (\ref{equation:sigmadef}) is that $E_k(\chi_1,\chi_2,k)$ has Fourier coefficients lying in the field $\bf R$ of real numbers only if $\chi_1,\chi_2$ are quadratic Dirichlet characters. Therefore, throughout the remainder of this paper all functions $\sigma_{\chi_1,\chi_2}^{k-1}$ will be associated to quadratic Dirichlet characters unless explicitly stated otherwise.

\section{The frequency of negative Fourier coefficients}

In this section, we answer some basic statistical questions about the sign of $\sigma_{\chi_1, \chi_2}^{k-1}$ using techniques from analytic number theory. First, we prove an elementary lemma, which allows us to reformulate our questions about $\sigma_{\chi_1, \chi_2}^{k-1}$ into questions about the behavior of $\chi_2$.  

\begin{lemma} \label{sgnlemma} If $(n, N) = 1$, then the sign of $\sigma_{\chi_1, \chi_2}^{k-1}(n)$ is completely determined by the behavior of $\chi_2(n).$ \end{lemma}

\begin{proof} Peeling off the $d = n$ term in the definition of $\sigma_{\chi_1, \chi_2}^{k-1}(n)$, we have \begin{align} \sigma_{\chi_1, \chi_2}^{k-1}(n) & = \sum_{d \mid n} \chi_1(n/d) \chi_2 (d) d^{k-1} \\ & = \chi_2(n)n^{k-1} + \sum_{\substack{d \mid n \\ d < n}} \chi_1(n/d)\chi_2(d)d^{k-1} \label{sgnlemma1}. \end{align} Factoring $n^{k-1}$ out of the sum in \eqref{sgnlemma1} yields \begin{align} \label{sgnlemma2} \chi_2(n) n^{k-1} + n^{k-1} \sum_{\substack{d \mid n \\ d < n}} \frac{\chi_1(n/d) \chi_2 (d)}{(n/d)^{k-1}} = n^{k-1}\left(\chi_2(n) + \sum_{\substack{d \mid n \\ d < n}} \frac{\chi_1(n/d) \chi_2(d)}{(n/d)^{k-1}}\right). \end{align} Since $|\chi_1(d)|, |\chi_2(d)| \leq 1$ for all $d \in \Z^+$, we have \begin{align*}  \sum_{\substack{d \mid n \\ d < n}} \frac{\chi_1(n/d) \chi_2(d)}{(n/d)^{k-1}} & \leq \sum_{\substack{d \mid n \\ d < n}} \frac{1}{(n/d)^{k-1}} \\ & \leq \sum_{m \geq 2} \frac{1}{m^{k-1}} \\ & = \zeta(k-1) - 1 \\ & < 1.\end{align*} Thus, the $\chi_2(n)$ term dominates in \eqref{sgnlemma2}, so we may conclude that $\mathrm{sgn} \ \sigma_{\chi_1, \chi_2}^{k-1}(n)  = \chi_2(n).$ \end{proof}

It follows from the definition of the Dirichlet character that $\chi_2(n) = -1$ half of the time and $\chi_2(n) = 1$ half of the time. This allows us to deduce a simple corollary from Lemma \ref{sgnlemma}.

\begin{corollary}\label{corollary:densityresult} For infinitely many integers $n$, we have $$\sigma_{\chi_1, \chi_2}^{k-1}(n) > 0.$$ Similarly, we have $$\sigma_{\chi_1, \chi_2}^{k-1}(n) < 0$$ for infinitely many $n$. In fact, we have $$\lim_{x \rightarrow \infty} \frac{1}{x} \cdot \#\{n \leq x : (n, N) = 1, \sigma_{\chi_1, \chi_2}^{k-1}(n) > 0\} = \frac{1}{2}$$ and $$\lim_{x \rightarrow \infty} \frac{1}{x} \cdot \#\{n \leq x : (n, N) = 1, \sigma_{\chi_1, \chi_2}^{k-1}(n) < 0\} = \frac{1}{2}.$$ \end{corollary}

\begin{rmk}
Although Corollary \ref{corollary:densityresult} shows that, asymptotically, the sign of $\sigma_{\chi_1,\chi_2}^{k-1}$ is positive as often as it is negative, it turns out when restricted to initial intervals of primes, $\sigma_{\chi_1,\chi_2}^{k-1}$ is negative more often than it is positive. This phenomenon is closely related to the well-studied {\em prime number races} in which certain arithmetic progressions can be shown to contain an unexpected number of primes less than some fixed $x$. Let $x\leq X$ and consider the character sum \begin{equation}\label{equation:charsum} \sum_{p\leq e^x} \chi_2(p).\end{equation} Lemma \ref{sgnlemma} shows that there are more primes in the interval $[1,e^x]$ for which $\sigma_{\chi_1,\chi_2}^{k-1}(p)<0$ than there are for which $\sigma_{\chi_1,\chi_2}^{k-1}(p)>0$ if and only if (\ref{equation:charsum}) is negative. Building on the seminal work of Rubinstein and Sarnak \cite{rubinstein-sarnak}, Fiorilli and Martin \cite[Section 3.6]{fiorilli-martin} have shown that under the assumption of GRH and LI (the linear independence hypothesis), the natural density of the set of $x$ such that the sum in (\ref{equation:charsum}) is negative exists and is equal to $\left(\frac{1}{2}+ (1+o(1))\sqrt{\frac{2}{\pi \log(N_2)}}\right)$. In particular this shows that there is always a bias towards $\sigma_{\chi_1,\chi_2}^{k-1}$ being negative, but that this bias dissipates as $N_2\rightarrow \infty$.

\end{rmk}

\section{The first negative Fourier coefficient}

In addition to determining how often $\sigma_{\chi_1, \chi_2}^{k-1}$ takes on a particular sign, one could also describe the location of the first sign change. For any fixed pair of characters, it is not difficult to obtain an upper bound for when the first sign change occurs; doing so is a straightforward application of Burgess' estimates (see, in particular, \cite[eq. (1.22)]{norton}). 

%\begin{thm}[Burgess]\label{burgessbound} Let $d$ denote the least quadratic non-residue modulo $N$. Then $d = O(N^\alpha)$ as $N \rightarrow \infty$, for any fixed $\alpha > \frac{1}{4 \sqrt{e}}.$ \end{thm}

\begin{thm}\label{thm:smallestsignchange} Let $p_0$ be the first integer co-prime to $N$ corresponding to the first sign change of $\sigma_{\chi_1, \chi_2}^{k-1}$. Then, for any fixed $\varepsilon > 0$, we have $$p_0 \ll_\varepsilon N^{\frac{1}{4\sqrt{e}} + \varepsilon}.$$ \end{thm}

\begin{proof} Let $\chi_2$ be a quadratic character modulo $N_2$. From Lemma \ref{sgnlemma}, we can write $\mathrm{sgn} \ \sigma_{\chi_1, \chi_2}^{k-1}(n) = \prod_{p^\ell \mid \mid n} \chi_2(p)^\ell.$ Since $\chi_2$ is a multiplicative function, $\sigma_{\chi_1, \chi_2}^{k-1}(n)$ will first be negative at a prime; i.e., there exists a prime $p_0$ for which $$\mathrm{sgn} \ \sigma_{\chi_1, \chi_2}^{k-1}(p_0) = \chi_2(p_0) = -1,$$ where $\mathrm{sgn} \ \sigma_{\chi_1, \chi_2}^{k-1}(p) = \chi_2(p) = 1$ for all $p < p_0.$ If $N = N_1 N_2$, then $\chi_2$ induces a quadratic character $\psi_2 \pmod{N}.$ The result follows from applying the Burgess bound for the character $\psi_2$ with $d = p_0$. \end{proof} 

Instead of fixing $\chi_1$ and $\chi_2$, we could consider the more difficult problem of determining when the first sign change occurs on average as we vary over $\chi_1, \chi_2$ (and remove the condition that $(n, N) =1).$ Our work for the remainder of the section will be to provide evidence for what this average should be. Along the way, we will prove an analogue of Kowalski, Lau, Soundararajan and Wu's theorem \cite{kowalskietal} on the frequency with which the signs of $\sigma_{\chi_1, \chi_2}^{k-1}$ agree with a fixed sequence of signs. 

Throughout this section, let $D, D_1, D_2$ represent fundamental discriminants corresponding to real quadratic fields. We will need the following bound for the count of fundamental discriminants with $|D| \leq x$, which can be obtained via a standard sieving argument: \begin{align}\label{Dmagnitude} \#\{D: |D| \leq x\} \sim \frac{1}{\zeta(2)} \cdot x.\end{align}  %Given a sequence of signs $\varepsilon_p = \{0, \pm 1\}$ with $p$ prime, we would like to estimate \begin{equation}\label{D1D2} \#\{(D_1, D_2) : |D_1 D_2| \leq x, \ \mathrm{sgn} (\sigma_{\chi_1, \chi_2}^{k-1}(p_i))= \varepsilon_{p_i} \ \hbox{for} \ 1 \leq i \leq k\}.\end{equation}

We will also make use of the following lemma.

\begin{lemma}\label{probabilities} Let $\mathbb{P}(\varepsilon, p)$ denote the proportion of fundamental discriminants $D$ with $\left(\frac{D}{p}\right) = \varepsilon.$ Then, we have

\begin{equation*} 
\mathbb{P}(\varepsilon, p) = 
     \begin{cases}
      \frac{p}{2p+2} & \mathrm{if} \ \varepsilon = 1\\
      \frac{p}{2p+2} & \mathrm{if} \ \varepsilon = -1\\
       \frac{1}{p+1} & \mathrm{if} \ \varepsilon = 0.
     \end{cases}
\end{equation*} 
\end{lemma}

\begin{proof} First, consider the case where $p$ is odd. Since the odd part of $D$ must be squarefree, each $D$ must lie in one of $p^2 - 1$ residue classes $\! \! \pmod{p^2}$. Moreover, the same sieving argument that allows us to count quadratic discriminants shows that the $D$'s are equidistributed among the residue classes $\! \! \pmod{p^2}$. Now, $\left(\frac{D}{p}\right) = 0$ if and only if $p \mid D$, which will occur in precisely $p-1$ of these residue classes. The remaining $p(p-1)$ residue classes are split equally between $+1$ and $-1$. In the case where $p =2$, we observe that the fundamental discriminants must lie in the residue classes $1, 5, 8, 9, 12, 13 \pmod{16}$ and the values $0, \pm 1$ all occur with the same likelihood. \end{proof}

%Below, we present the proof of our main analytic result.

%\begin{thm}\label{mainanalytic} For any fixed sequence of signs $\{\varepsilon_1,...,\varepsilon_k\}$, we have \begin{align*}\#\{(D_1, D_2) : |D_1D_2| \leq x, \ \mathrm{sgn}(\sigma_{\chi_1, \chi_2}^{k-1}(p_i)) = \varepsilon_i, \ \mathrm{for} \ 1 \leq i \leq k\} \\ \sim \frac{1}{(\zeta(2))^2} \left(\prod_{i=1}^k\mathbb{P}(\varepsilon_i, p_i)\right)\left(\prod_{\substack{\varepsilon_i = 0 \\ 1 \leq i \leq k}} \frac{1}{p_i+1}\right)\left(\prod_{\substack{\varepsilon_i \neq 0 \\ 1 \leq i \leq k}}\left(1 + \frac{1}{p_i+1}\right)\right) x \log x.\end{align*} \end{thm} 

\begin{thm}\label{mainanalytic} Let $\mathcal{D} = \{(D_1, D_2) : |D_1 D_2| \leq x\}.$ Given a sequence of signs $\varepsilon_{p_1},...,\varepsilon_{p_\ell}$ with $\varepsilon_{p_i} \in \{0, \pm 1\}$ for $1 \leq i \leq \ell$, we have \begin{multline*} \frac{1}{|\mathcal{D}|} \#\{(D_1, D_2) \in \mathcal{D}:  \mathrm{sgn} \ \sigma_{\chi_1, \chi_2}^{k-1}(p_i) = \varepsilon_{p_i}, 1 \leq i \leq \ell\} \\ \to \bigg(\prod_{\substack{\varepsilon_{p_i} = 0 \\ 1 \leq i \leq \ell}} \frac{1}{(p_i+1)^2}\bigg) \bigg(\prod_{\substack{\varepsilon_{p_i} \neq 0 \\ 1 \leq i \leq \ell}}\frac{p_i(p_i+2)}{2(p_i+1)^2}\bigg), \end{multline*} as $x \rightarrow \infty.$ \end{thm}

\begin{proof} Since $$\sigma_{\chi_1, \chi_2}^{k-1}(p) = \sum_{d \mid p} \chi_1(p/d) \chi_2(d) d^{k-1} = \chi_1(p) + \chi_2(p)p^{k-1},$$ we have \begin{equation*}
\mathrm{sgn} \ \sigma_{\chi_1, \chi_2}^{k-1}(p) = 
     \begin{cases}
       \chi_1(p) & \mathrm{if} \ p \mid D_2, \\
       \chi_2(p) & \mathrm{if} \ p \nmid D_2.
     \end{cases}
\end{equation*} 

Let $\mathcal{P} = \{p_i : p_i \mid D_2, 1 \leq i \leq \ell\}$ and let $\mathcal{Q} = \{p_1,...,p_\ell\}\setminus \mathcal{P}.$ For each fixed $\mathcal{P}, \mathcal{Q}$, we will count pairs $\chi_1, \chi_2$ with $\chi_1(p) = \varepsilon_p$ for all $p \in \mathcal{P}$ and $\chi_2(p) = \varepsilon_p$ for all $p \in \mathcal{Q}$, subject to the conditions that $\prod_{p \in \mathcal{P}} p \mid D_2$ and $\left(\frac{D_2}{\prod_{p \in \mathcal{P}} p}, p_1 \cdots p_\ell\right) = 1$. Since $\mathcal{P} \bigsqcup \mathcal{Q} = \{p_1,...,p_\ell\}$ then our choice of $\mathcal{P}$ completely determines our choice of $\mathcal{Q}$. As a result, estimating the number of pairs $(D_1, D_2) \in \mathcal{D}$ with $\mathrm{sgn} \ \sigma_{\chi_1, \chi_2}^{k-1}(p_i) = \varepsilon_{p_i}$ amounts to estimating the following sum: \begin{align}\label{triplesumpq} \sum_{\substack{\mathcal{P} \subseteq \{p_1,...,p_\ell\} \\ \varepsilon_p \neq 0 \ \mathrm{if} \ p \not\in \mathcal{P}}} \sum_{\substack{|D_1| \leq x \\ \chi_1(p) = \varepsilon_p  \\ \mathrm{if} \ p \in \mathcal{P}}} \sum_{\substack{|D_2| \leq x/|D_1| \\ \chi_2(p) = \varepsilon_p \\ \mathrm{for} \ p \in \mathcal{Q} \\ \chi_2(p) = 0 \ \mathrm{if} \ p \in \mathcal{P}}} 1. \end{align} 

We can imitate the proof of \cite[(5.3)]{pollack} to estimate the inner sum for \eqref{triplesumpq} using the probabilities obtained in Lemma \ref{probabilities} along with an important result of Wood \cite[Theorem 1.3]{wood}, which tells us that the probabilities $\mathbb{P}(\varepsilon_p, p)$ are independent. This allows us to use \eqref{Dmagnitude} to re-write \eqref{triplesumpq} (ignoring the error term, which is negligible relative to the size of the main term) as \begin{align}\label{doublesumpq} \frac{x}{\zeta(2)} \sum_{\substack{\mathcal{P} \subseteq \{p_1,...,p_\ell\} \\ \varepsilon_p \neq 0 \\ \mathrm{if} \ p \notin \mathcal{P}}} \left(\prod_{p \not\in \mathcal{P}} \mathbb{P}(\varepsilon_p, p)\right) \left(\prod_{p \in \mathcal{P}} \frac{1}{p+1}\right) \sum_{\substack{|D_1| \leq x \\ \chi_1(p) = \varepsilon_p \\ \mathrm{for} \ p \in \mathcal{P}}} \frac{1}{|D_1|}.\end{align} %%+ ET = O((x/|D_1|)^{3/4})

We can use partial summation to estimate the inner sum, taking $A(x) = \sum_{|D_1| \leq x, \chi_1(p) = \varepsilon_p} 1$ and $f(x) = \frac{1}{x}.$ For the sake of brevity, let $C = \frac{1}{\zeta(2)} \prod_{p \in \mathcal{P}} \mathbb{P}(\varepsilon_p, p).$ By partial summation, we have \begin{align*}\sum_{\substack{|D_1| \leq x \\ \chi_1(p) = \varepsilon_p \\ \mathrm{for} \ p \in \mathcal{P}}} \frac{1}{|D_1|} & = (A(x)f(x) - A(1)f(1)) - \int_1^x A(t) f'(t) dt \\ & \sim C \log x.\end{align*} Thus, we can re-write \eqref{doublesumpq} in the following manner: \begin{align}\label{finalsumexpression}\frac{x \log x}{(\zeta(2))^2} \left(\prod_{i=1}^\ell \mathbb{P}(\varepsilon_{p_i}, p_i)\right) \sum_{\substack{\mathcal{P} \subseteq \{p_1,..., p_\ell\} \\ \varepsilon_p = 0 \Rightarrow p \in \mathcal{P}}} \left(\prod_{p \in \mathcal{P}} \frac{1}{p+1}\right).\end{align} %%ET from partial summation: O(1)

To estimate the final sum, let $\mathcal{P}_0 = \{p : \varepsilon_p = 0\}$ and let $\mathcal{P}_1 = \mathcal{P}\setminus \mathcal{P}_0.$ Then \begin{align*}\sum_{\substack{\mathcal{P} \subseteq \{p_1,...,p_\ell\} \\ \varepsilon_p = 0 \Rightarrow p \in \mathcal{P}}} \left(\prod_{p \in \mathcal{P}} \frac{1}{p+1}\right) & = \sum_{\mathcal{P}_0 \subseteq \mathcal{P} \subseteq \{p_1,...,p_\ell\}} \left(\prod_{p \in \mathcal{P}} \frac{1}{p+1}\right) \\ & = \sum_{\mathcal{P}_1 \subseteq \{p_1,...,p_\ell\}\setminus \mathcal{P}_0} \left(\prod_{p \in \mathcal{P}_0} \frac{1}{p+1}\right)\left(\prod_{p \in \mathcal{P}_1} \frac{1}{p + 1}\right) \\ & = \left(\prod_{p \in \mathcal{P}_0} \frac{1}{p+1}\right) \sum_{\mathcal{P}_1 \subseteq \{p_1,...,p_\ell\}\setminus \mathcal{P}_0} \left(\prod_{p \in \mathcal{P}_1} \frac{1}{p +1}\right) \\ & = \left(\prod_{p \in \mathcal{P}_0} \frac{1}{p+1}\right) \left(\prod_{p \in \mathcal{P}_1} \left(1 + \frac{1}{p+1}\right)\right) \\ & = \left(\prod_{\substack{\varepsilon_{p_i} = 0 \\ 1 \leq i \leq \ell}} \frac{1}{p_i + 1}\right) \left(\prod_{\substack{\varepsilon_{p_i} \neq 0 \\ 1 \leq i \leq \ell}} \left(1 + \frac{1}{p_i+1}\right)\right).\end{align*} Since the number of fundamental discriminants with $|D| \leq x$ is asymptotic to $x/\zeta(2)$ from \eqref{Dmagnitude}, then the number of pairs $(D_1, D_2)$ with $|D_1 D_2| \leq x$ is asymptotic to $(x \log x)/\zeta(2)^2.$ Our result follows after dividing \eqref{finalsumexpression} by $|\mathcal{D}| \sim \frac{x \log x}{\zeta(2)^2}$ and applying Lemma \ref{probabilities}. \end{proof}

We can use Theorem \ref{mainanalytic} to give the following conjecture for when, on average, the first negative value of $\mathrm{sgn}(\sigma_{\chi_1, \chi_2}^{k-1}(p_i))$ occurs.

\begin{conj} Let $\eta(D_1, D_2)$ represent the smallest $p_i$ for which $\mathrm{sgn}(\sigma_{\chi_1, \chi_2}^{k-1}(p_i)) = -1.$ As $x \rightarrow \infty$, one would expect $$\frac{\sum_{|D_1D_2| \leq x} \eta(D_1, D_2)}{\sum_{|D_1D_2| \leq x} 1} \rightarrow \theta,$$ where $$\theta:= \sum_{\ell=1}^\infty \frac{p_\ell^2(p_\ell+2)}{2(p_\ell + 1)^2} \prod_{i=1}^{\ell-1} \frac{2 + p_i(p_i + 2)}{2(p_i + 1)^2}.$$ Numerically, $$\theta = 	
3.9750223902667539847734759105175510246019355513991....$$ \end{conj} 

\begin{rmk} Heuristically, one would expect \begin{align*}\frac{\sum_{|D_1D_2| \leq x} \eta(D_1, D_2)}{\sum_{|D_1D_2| \leq x} 1} & = \sum_{\ell = 1}^\infty p_\ell \cdot \mathrm{Prob}\left(\eta(D_1, D_2) = p_\ell\right) \\ & = \sum_{\ell=1}^\infty p_\ell  \cdot \mathrm{Prob}(\varepsilon_{p_\ell} = -1) \cdot \prod_{i = 1}^{\ell-1} \mathrm{Prob}(\varepsilon_{p_i} = 0 \ \mathrm{or} \ 1) \\ & = \sum_{\ell = 1}^\infty p_\ell \cdot \frac{p_\ell(p_\ell+2)}{2(p_\ell+1)^2} \cdot \prod_{i=1}^{\ell-1} \left(\frac{1}{(p_i+1)^2} + \frac{p_i(p_i+2)}{2(p_i+1)^2}\right),\end{align*} where the final equality follows from Theorem \ref{mainanalytic}. The conjectural result can be obtained after some simplification. It may be possible to make this argument rigorous using a large sieve argument (see, for example, \cite{pollack} or \cite{erdos}).  \end{rmk}

\section{A strong multiplicity-one theorem and applications}

In this section, we prove a strong multiplicity-one theorem for Eisenstein series newforms and provide several applications that clarify the extent to which the sign of the Hecke eigenvalues of an Eisenstein series newform determine the newform. We note that while the results of this section are stated for classical elliptic Eisenstein series, all of our proofs hold, mutatis mutandis, for adelic Hilbert modular Eisenstein series. We have chosen to state our results for classical elliptic Eisenstein series in order to maintain the cohesion of the paper and because we feel that the technicalities needed to define adelic Hilbert modular forms would obscure the ideas underlying our proofs. We will simply mention that all of the relevant newform theory for Hilbert modular forms was proven by Shemanske and Walling \cite{shemanske-walling} (for cusp forms) and Atwill and the first author \cite{linowitz-atwill} (for Eisenstein series).

Throughout the remainder of this paper we adopt the convention that if $f\in M_k(N,\chi)$ is a modular form then the $n^{th}$ Fourier coefficient of $f$ is denoted $a_f(n)$. In the event that $f$ is a normalized Hecke eigenform, we note that the eigenvalue of $f$ with respect to the Hecke operator $T_n$ (where $(n,N)=1$) is equal to $a_f(n)$.

We begin by proving a strong multiplicity-one theorem which shows that an Eisenstein series newform is uniquely determined by its Hecke eigenvalues for any set of primes having density $\delta > 1/2$. 

\begin{thm}\label{thm:strongmultone}
	Let $f\in E_k(N,\chi_f)$ and $g\in E_{k^\prime}(N^\prime,\chi_g)$ be newforms such that $$a_f(p)=a_g(p)$$ for a set $S$ of primes with $\delta(S)>1/2$. Then $k=k^\prime$, $N=N^\prime$, $\chi_f=\chi_g$ and $f=g$.  
\end{thm}

\begin{rmk}
Theorem \ref{thm:strongmultone} generalizes a result of Atwill and the first author \cite[Theorem 3.6]{linowitz-atwill} which, working in the Hilbert modular setting, considered the special case in which $k=k^\prime$ and $\chi_f=\chi_g$.
\end{rmk}

Our proof of Theorem \ref{thm:strongmultone} will make use of the following lemma.

\begin{lemma}\label{lemma:strongmultonelemma}
Let $\chi_1,\chi_2,\psi_1,\psi_2$ be Dirichlet characters modulo $M$ and $c$ be a nonzero complex number. There exists a constant $p_0$ such that if $p>p_0$ is prime and $$\chi_1(p)+\chi_2(p)p^{k-1}=c\left(\psi_1(p)+\psi_2(p)p^{k-1}\right),$$ then $\chi_1(p)=c\psi_1(p)$ and $\chi_2(p)=c\psi_2(p)$.
\end{lemma}
\begin{proof}
Suppose that $p\nmid M$ and $\chi_1(p)+\chi_2(p)p^{k-1}=c\left(\psi_1(p)+\psi_2(p)p^{k-1}\right)$. If $\chi_2(p)=c\psi_2(p)$ then $\chi_1(p)=c\psi_1(p)$. Otherwise, we have 
\begin{equation}\label{equation:strongmultoneequation}
p^{k-1}=\frac{\chi_1(p)-c\psi_1(p)}{c\psi_2(p)-\chi_2(p)}.
\end{equation} 
Because $c$ is fixed, the absolute value of the right hand side of (\ref{equation:strongmultoneequation}) can be bounded independently of $p$, which leads to a contradiction for all sufficiently large primes $p$.\end{proof}

We now prove Theorem \ref{thm:strongmultone}.

\begin{proof}
Write $f=E(\chi_1,\chi_2,k)$ and $g=E(\psi_1,\psi_2,k^\prime)$. We begin by proving that $k=k^\prime$. Suppose that $k\neq k^\prime$ and without loss generality that $k>k^\prime$. Because $a_f(p)=\sigma_{\chi_1,\chi_2}^{k-1}(p)=\chi_1(p)+\chi_2(p)p^{k-1}$, it is clear that there exists $\varepsilon_f \in (0,1/2)$ such that for all sufficiently large primes $p$, $|a_f(p)|>p^{k-1-\varepsilon_f}$. Similarly, there exists $\varepsilon_g\in(0,1/2)$ such that for all sufficiently large primes $p$, $|a_g(p)|< p^{k^\prime-1+\varepsilon_g}$. As $\delta(S)>1/2$, we can select a prime $p$ in $S$ such that $a_f(p)=a_g(p)$, $|a_f(p)|>p^{k-1-\varepsilon_f}$ and $|a_g(p)|< p^{k^\prime-1+\varepsilon_g}$. It follows that for such a prime $p$, $p^{k^\prime-1+\varepsilon_g}>p^{k-1-\varepsilon_f}$. On the other hand it is easy to show that $k\geq k^\prime+1$ implies $p^{k-1-\varepsilon_f}>p^{k^\prime-1+\varepsilon_g}$. This contradiction proves that $k=k^\prime$.

We now show that $\chi_1=\psi_1$ and $\chi_2=\psi_2$. Denote by $\chi_1^\prime,\chi_2^\prime, \psi_1^\prime, \psi_2^\prime$ the induced Dirichlet characters modulo $NN^\prime$. Let $p_1,\dots,p_{\varphi(NN^\prime)}$ represent the residue classes of $(\Z/NN^\prime\Z)^\times$ and assume that all of the primes $p_i$ are large enough so that Lemma \ref{lemma:strongmultonelemma} holds with $c=1$. If $a_f(p_i)=a_g(p_i)$ then, by Lemma \ref{lemma:strongmultonelemma}, $\chi_1^\prime(p_i)=\psi_1^\prime(p_i)$ and $\chi_2^\prime(p_i)=\psi_2^\prime(p_i)$. If this occurs for $s$ values of $i$, then the density of primes $p$ for which $\chi_1(p)+\chi_2(p)p^{k-1}=\psi_1(p)+\psi_2(p)p^{k-1}$ is at most $\frac{s}{\varphi(NN^\prime)}$. It follows that this occurs for more than $\frac{\varphi(NN^\prime)}{2}$ values of $i$. The orthogonality relations now show that $\chi_1^\prime\bar{\psi_1^\prime}$ is the principal character, hence the primitive characters $\chi_1,\psi_1$ inducing $\chi_1^\prime,\psi_1^\prime$ are equal as well. An identical argument shows that $\chi_2=\psi_2$. 

Having shown that $\chi_1 = \psi_1$ and $\chi_2 = \psi_2$ we observe that $N = \cond(\chi_1) \cdot \cond(\chi_2) = \cond(\psi_1) \cdot \cond(\psi_2) = N^\prime$ and $\chi_f = \chi_1\chi_2 = \psi_1\psi_2 = \chi_g$, hence $f = E(\chi_1,\chi_2,k) = E(\psi_1,\psi_2,k^\prime) = g$.\end{proof}

It is easy to see that Theorem \ref{thm:strongmultone} is best possible in the sense that there exist distinct newforms whose Hecke eigenvalues differ at a set of primes having density equal to $1/2$. Examples of this form may be constructed by considering a newform $f\in E_k(N,\chi_f)$ and its twist by a quadratic character of conductor relatively prime to $N$ (the fact that such a character twist will still be a newform follows from the results of \cite[Section 5]{linowitz-atwill}). In light of this, it is natural to ask if twisting by a quadratic character is the only way produce examples of this form. We address this question in the following theorem.

\begin{thm}\label{thm:quadtwists}
Let $f\in E_k(N,\chi_f)$ and $g\in E_{k^\prime}(N^\prime,\chi_g)$ be distinct newforms such that $$a_f(p)=a_g(p)$$ for a set $S$ of primes with $\delta(S)=1/2$. Then there exists a quadratic character $\theta$ such that $$a_f(p)=\theta(p)a_g(p)$$ for all primes $p\nmid NN^\prime$.
\end{thm}

\begin{proof}
By Lemma \ref{lemma:strongmultonelemma}, there exists a subset $S^\prime$ of $S$ such that $\delta(S^\prime)=1/2$ and, for all primes $p\in S^\prime$, we have $\chi_1(p)=\psi_1(p)$ and $\chi_2(p)=\psi_2(p)$. Let $i\in\{ 1,2\}$ and consider $\theta_i=\chi_i\bar{\psi_i}$. From above, $\theta_i(n)=1$ for $1/2$ of the elements in $\left(\Z/NN^\prime\Z\right)^\times$. If either $\theta_1$ or $\theta_2$ is principal, then our hypothesis that $a_f(p)=a_g(p)$ implies $f=g$, which would be a contradiction. Otherwise the orthogonality relations show that $\theta_i(n)=-1$ whenever $\theta_i(n)\neq 1$. Therefore, $\theta_1$ and $\theta_2$ are quadratic characters. Moreover, it must be the case that $\theta_1=\theta_2$, since $\theta_1(p)=1=\theta_2(p)$ for all primes $p\in S^\prime$, hence both must assume the value $-1$ on all primes $p\nmid NN^\prime$ that lie in the complement of $S^\prime$. The theorem follows.\end{proof}

Having shown that an Eisenstein series newform is uniquely determined by its Hecke eigenvalues for any set of primes with density greater than $1/2$, we now show that in fact a stronger statement is true: an Eisenstein series newform is uniquely determined by the \textit{signs} of Hecke eigenvalues for any set of primes with density greater than $1/2$. Here, as usual, we adopt the convention that the sign $\mathrm{sgn}(z)$ of a nonzero complex number $z$ is equal to $\frac{z}{|z|}$ (i.e. $\mathrm{sgn}(z)$ is the point on the complex unit circle closest to $z$). This result complements Theorem 4 of Kowalksi, Lau, Soundararajan and Wu \cite{kowalskietal}, which shows that a similar result holds for cuspidal newforms.

\begin{thm}\label{thm:sgnmultone}
Let $f\in E_k(N,\chi_f)$ and $g\in E_k(N^\prime,\chi_g)$ be newforms such that $$\mathrm{sgn}(a_f(p))=\mathrm{sgn}(a_g(p))$$ for a set $S$ of primes with $\delta(S)>1/2$. Then $N=N^\prime$, $\chi_f=\chi_g$ and $f=g$.
\end{thm}

\begin{rmk}
As was the case with Theorem \ref{thm:strongmultone}, our Theorem \ref{thm:sgnmultone} is best possible. Indeed, if $f\in E_k(N,\chi_f)$ is a newform and $\theta$ is a primitive quadratic Dirichlet character whose conductor is relatively prime to $N$, then the twist $f_\theta$ of $f$ by $\theta$ will be a newform whose Hecke eigenvalues have the same sign as those of $f$ for a set of primes of density equal to $1/2$ (the set of primes $p$ for which $\theta(p)=1$).
\end{rmk}

Our proof of Theorem \ref{thm:sgnmultone} will make use of the following lemma, which follows immediately from the case where $m = 2$.

\begin{lemma}\label{lemma:signs}
Let $z_1,\dots,z_m$ be distinct complex numbers lying on the unit circle. Then there exists an $\varepsilon>0$ such that for any positive real number $r$ and $i\neq j$ we have $$|rz_i-z_j|>\varepsilon.$$
\end{lemma}
%\begin{proof}
%It suffices to prove the lemma for $m=2$, in which case the proof follows immediately.
%\end{proof}

We now prove Theorem \ref{thm:sgnmultone}. 

\begin{proof}
Write $f=E(\chi_1,\chi_2,k)$ and $g=E(\psi_1,\psi_2,k^\prime)$. Let $n_1,\dots,n_{\varphi(NN^\prime)}$ represent the residue classes of $(\Z/NN^\prime\Z)^\times$ and $\varepsilon$ be the constant from Lemma \ref{lemma:signs} applied to the set $$\{ \chi_i(n_j) : i\in\{1,2\}, 1\leq j\leq \varphi(NN^\prime) \} \bigcup \{ \psi_i(n_j) : i\in\{1,2\}, 1\leq j\leq \varphi(NN^\prime) \}.$$
Let $S^\prime\subset S$ be the subset of $S$ consisting of primes $p$ for which $2p^{1-k}<\varepsilon$ and $p>NN^\prime$. Note that $\delta(S^\prime)=\delta(S)>1/2$. 

For each prime $p\in S^\prime$, define a positive real number $\varepsilon_p:=\frac{a_f(p)}{a_g(p)}=\frac{|a_f(p)|}{|a_g(p)|}$. We claim that $\varepsilon_p=1$ for all primes $p\in S^\prime$. If $\varepsilon_p\neq 1$ for some prime $p$, then by interchanging $f$ and $g$ (if necessary) we may assume $\varepsilon_p<1$. By definition of $\varepsilon_p$ we have, $$\chi_1(p)+\chi_2(p)p^{k-1}=\varepsilon_p\psi_1(p)+\varepsilon_p\psi_2(p)p^{k-1}.$$ From this identity, it follows that
\begin{align*}
|\varepsilon_p\psi_2(p)-\chi_2(p)|&=\frac{|\chi_1(p)-\varepsilon_p\psi_1(p)|}{p^{k-1}}\\ &\leq \frac{1+\varepsilon_p}{p^{k-1}}\\ &< 2p^{1-k} \\ &<\varepsilon.\end{align*}

This contradicts Lemma \ref{lemma:signs}, hence $\varepsilon_p=1$. Theorem \ref{thm:sgnmultone} now follows from Theorem \ref{thm:strongmultone}.\end{proof}

We conclude this section by proving that up to twisting by a Dirichlet character, an Eisenstein series newform is determined by the $n^{th}$ powers of its Hecke eigenvalues for any set of primes with density greater than $1/2$.

\begin{thm}\label{thm:nthpowers}
Let $f\in E_k(N,\chi_f)$ and $g\in E_{k^\prime}(N^\prime,\chi_g)$ be newforms such that for some integer $n\geq 2$, $$a_f(p)^n=a_g(p)^n$$ holds for a set $S$ of primes with $\delta(S)>1/2$. Then $k=k^\prime$ and there exists a Dirichlet character $\theta$ with $\cond(\theta) \mid NN^\prime$ such that $\chi_f=\theta^2\chi_g$ and $$a_f(p)=\theta(p)a_g(p)$$ for all primes $p\nmid NN^\prime$.
\end{thm}
\begin{proof}
As before, write $f=E(\chi_1,\chi_2,k)$ and $g=E(\psi_1,\psi_2,k^\prime)$. The fact that $k=k^\prime$ can be proven using the ideas employed in the proof of the corresponding result in Theorem \ref{thm:strongmultone}. We therefore omit this proof.

If $p\nmid NN^\prime$ then assign, for each prime $p\in S$, the $n$th root of unity $\varepsilon_p$ such that $a_f(p)=\varepsilon_p a_g(p)$. By Lemma \ref{lemma:strongmultonelemma}, there exists a subset $S^\prime$ of $S$ such that $\delta(S^\prime)=\delta(S)$ and, for all $p\in S^\prime$, we have $\chi_1(p)=\varepsilon_p\psi_1(p)$ and $\chi_2(p)=\varepsilon_p\psi_2(p)$. 

Let $\theta=\chi_1\bar{\psi_1}$ and consider the twist $g_\theta$ of $g$ by $\theta$. Since $g_\theta$ is a simultaneous Hecke eigenform for all primes $p\nmid NN^\prime$, there exists a newform equivalent to $g_\theta$. It follows from Theorem \ref{thm:strongmultone} that this newform is $f$, hence $\chi_f=\theta^2\chi_g$. Let $p\nmid NN^\prime$ be prime. Then computing the $p$th Hecke eigenvalues of $g_\theta$ and $f$ shows that $$\theta(p)(\psi_1(p)+\psi_2(p)p^{k-1})=\chi_1(p)+\chi_2(p)p^{k-1},$$ which finishes our proof.
\end{proof}

The following is an immediate consequence of Theorem \ref{thm:nthpowers}.

\begin{corollary}
Let the notation be as in Theorem \ref{thm:nthpowers}. If $\chi_f=\chi_g$, then $\theta$ is quadratic and $a_f(p)^2=a_g(p)^2$ for all primes $p\nmid NN^\prime$.
\end{corollary}

\section{Eisenstein series with rational and non-negative Fourier coefficients}

\subsection{The field of Fourier coefficients of an Eisenstein newform}

Let $f\in S_k(N,\chi)$ be a cuspidal newform, let $\Q(f)$ be the field generated by its Fourier coefficients and recall that $\Q(f)$ is either a totally real number field or a CM field (that is, a totally imaginary quadratic extension of a totally real number field) \cite[Proposition 3.2]{ribet}. We shall show that the same is true for the field of Fourier coefficients of an Eisenstein newform. Our proof will rely on the following two basic properties of CM fields, the second of which is an easy consequence of the first.

\begin{lemma}\label{lemma:cm1}\cite[Proposition 5.11]{shimura-book}
Let $F$ be a number field and $\rho\in\Aut(\C)$ be complex conjugation. Then $F$ is a CM field if and only if
\begin{enumerate}
\item $\rho$ induces a nontrivial automorphism of $F$.
\item Every isomorphism $\sigma$ of $F$ into $\C$ commutes with $\rho$.
\end{enumerate}	
\end{lemma}

\begin{lemma}\label{lemma:cm2}
Let $F$ be a CM field and $F^\prime$ be a subfield of $F$. Then $F^\prime$ is either a totally real number field or a CM field.
\end{lemma}

%\begin{example}\label{example:cm3}
%Let $n>2$ and $\bf{F}=\Q(\zeta_n)$. Then $\Q(\zeta_n+\zeta_n^{-1})$ is the maximal real subfield of $\bf{F}$ and $\Q(\zeta_n)$ may be obtained from $\Q(\zeta_n+\zeta_n^{-1})$ by adjoining a square root of $\zeta_n^2 + \zeta_n^{-2}-2$, which is totally negative.
%\end{example}

\begin{prop}\label{prop:heckefield}
Let $E=E(\chi_1,\chi_2,k)\in E_k(N,\chi)$ be a newform and $\Q(E)$ be the field obtained by adjoining to $\Q$ all of the Fourier coefficients of $E$.
\begin{enumerate}
	\item $\Q(E)$ has finite degree over $\Q$.
	\item $\Q(E)$ is a Galois extension of $\Q$ with abelian Galois group.
	\item $\Q(E)$ contains $\chi(n)$ for all $n\geq 1$.
	\item $\Q(E)$ is either a totally real number field or a CM field.
\end{enumerate}
\end{prop}
\begin{proof}
The formula for the Fourier coefficients of $E_{\chi_1,\chi_2}$ makes it clear that $\Q(E)$ is contained in $\Q(\zeta_N)$. This proves the first two assertions. Assertion (iii) follows from the identity $a_E(p)^2=a_E(p^2)+\chi(p)p^{k-1}$ and the fact that every residue class of $(\Z/N\Z)^\times$ is represented by a prime. We have already noted that $\Q(E)$ is contained in a cyclotomic field. If the latter field is totally real (hence equal to $\Q$) then $\Q(E)$ is totally real as well. Otherwise this cyclotomic field is a CM field and assertion (iv) follows from Lemma \ref{lemma:cm2}.
\end{proof}

The following is an immediate consequence of Proposition \ref{prop:heckefield}.

\begin{prop}\label{prop:rationalcoefficients}
Suppose that a newform $E\in E_k(N,\chi)$ has rational Fourier coefficients and is not identically zero. Then $\chi$ is either trivial or quadratic.
\end{prop}
\begin{proof}
It is clear that for such an $E$ we have $\Q(E)=\Q$. Proposition \ref{prop:heckefield} now shows that $\chi$ must be real-valued. The proof follows. \end{proof}

\begin{rmk}
An alternative proof of Proposition \ref{prop:rationalcoefficients} can be obtained by arguing that the $p^{th}$ Fourier coefficient $a_E(p)=\chi_1(p)+\chi_2(p)p^{k-1}$ of $E(\chi_1,\chi_2,k)$ is rational if and only if both $\chi_1(p)$ and $\chi_2(p)$ are rational (equivalently, $a_E(p)$ is a real number if and only if $\chi_1(p)$ and $\chi_2(p)$ are real numbers and thus lie in the set $\{0,\pm 1\}$). 
The proof then follows from the fact that $\chi=\chi_1\chi_2$.
\end{rmk}

Our next goal is to extend Proposition \ref{prop:rationalcoefficients} to arbitrary Eisenstein series. Before doing so we require some terminology.

Let $\sigma\in\Aut(\C)$ and $\chi$ be a character on a finite abelian group $G$. We will denote by $\chi^\sigma$ the character $$\chi^\sigma: G\rightarrow \C^\times,\qquad\qquad \chi^\sigma(g)=\sigma(\chi(g)).$$ 

It is well-known that, given an Eisenstein series $f\in E_k(N,\chi)$ and $\sigma\in \Aut(\C)$, there exists an Eisenstein series $f^\sigma\in E_k(N,\chi^\sigma)$ such that $a_{f^\sigma}(n)=\sigma(a_f(n))$ for all $n\geq 0$.

\begin{lemma}\label{lemma:autotwist}
Let $\sigma\in \Aut(\C)$ and $f\in E_k(N,\chi)$ be a newform. Then $f^\sigma\in E_k(N,\chi^\sigma)$ is a newform.
\end{lemma}
\begin{proof}
It is easy to see that if $f=E(\chi_1,\chi_2,k)$ then $f^\sigma=E(\chi_{1}^\sigma,\chi_{2}^\sigma,k)$:

\begin{align*}
a_{E(\chi_{1},\chi_{2},k)^\sigma}(n)=\sigma(a_{E(\chi_{1},\chi_{2},k)}(n)) &= \sigma\left(\sum_{d\mid n} \chi_{1}(n/d)\chi_{2}(d) d^{k-1}\right)	 \\ 
&= \sum_{d\mid n} \chi_{1}^\sigma(n/d)\chi_{2}^\sigma(d) d^{k-1}\\
&= a_{E(\chi_{1}^\sigma,\chi_{2}^\sigma,k)}(n).
\end{align*}
\end{proof}

Let $f\in E_k(N,\chi)$ be a newform with field of Fourier coefficients $\Q(f)$ and let $G=\Gal(\Q(f)/\Q)$. Given $\alpha\in\Q(f)$, we define the \emph{trace} $\Tr(\alpha f)$ of $\alpha f$ as $$\Tr(\alpha f)=\sum_{\sigma\in G} \sigma(\alpha)f^\sigma.$$ It follows immediately from the definition that the Fourier coefficients of $\Tr(\alpha f)$ are rational numbers. We also remark that if $\Q(f)\neq \Q$ then $\Tr(\alpha f)\in E_k(N,\chi)$ if and only if $\chi$ is real-valued (otherwise $\Tr(\alpha f)$ simply lies in the larger space $E_k(\Gamma_1(N))$).

\begin{thm}\label{thm:rationalcoefficients}
If $E\in E_k(N,\chi)$ has rational Fourier coefficients and is not identically zero, then $\chi$ is either trivial or quadratic. Moreover, $E$ can be written uniquely as a sum of shifts of traces of newforms in $E_k(N,\chi)$ and rational multiples of newforms $E(\chi_1,\chi_2,k)\in E_k(N,\chi)$, where $\chi_1$ and $\chi_2$ are real-valued Dirichlet characters.
\end{thm}
\begin{proof}
Fix an Eisenstein series $E\in E_k(N,\chi)$ with rational Fourier coefficients. Then $E$ has a unique decomposition \begin{equation}\label{equation:decomp}E=\sum_{i=1}^r c_i \left( E_i\mid B_{d_i}\right) \end{equation} as a linear combination of shifts of newforms. Here the $c_i$ are nonzero complex numbers, the $d_i$ are divisors of $N$ and $E_i=E(\chi_{i,1},\chi_{i,2},k)$ is a newform of level $N/d_i$ and character $\chi$.

Consider one of the summands $c_i \left( E_i\mid B_{d_i}\right)$ in equation (\ref{equation:decomp}). 

Suppose first that $E_i$ has rational Fourier coefficients. This is equivalent to assuming that $\chi_{i,1}$ and $\chi_{i,2}$ are real-valued. In this case if $\sigma\in \Aut(\C)$ then $E_i^\sigma=E_i$, hence $\left( E_i\mid B_{d_i}\right)=\left( E_i^\sigma\mid B_{d_i}\right)$. The uniqueness of the decomposition in equation (\ref{equation:decomp}) shows that $\sigma(c_i)=c_i$. Since $\sigma$ was an arbitrary element of $\Aut(\C)$, we conclude that $c_i\in\Q$.

Now suppose that $E_i$ does not have all rational Fourier coefficients or equivalently, $\chi_{i,1}$ and $\chi_{i,2}$ are not both real-valued and $\Q(E_i)$ is a nontrivial extension of $\Q$. By Proposition \ref{prop:heckefield}, the field $\Q(E_i)$ is a Galois extension of $\Q$. Let $\F$ be the compositum of the fields $\Q(E_i)$ as $i$ ranges over $1,\dots,r$ and let $G$ be the Galois group of $\F$ over $\Q$. 

We claim that $c_i$ (which \textit{a priori} is only known to be a nonzero complex number) is an element of $\Q(E_i)$. Indeed, if $\tau\in\Aut(\C/\Q(E_i))$, then $E_i^\tau=E_i$. That $\tau(c_i)=c_i$ now follows from the uniqueness of the decomposition in equation (\ref{equation:decomp}).

Note that, if $\sigma\in G$, then $E^\sigma=E\in E_k(N,\chi)$ because the Fourier coefficients of $E$ are rational numbers. On the other hand $E^\sigma=\sum_{i=1}^r \sigma(c_i) \left( E_i^\sigma \mid B_{d_i}\right)$, hence $E_i^\sigma\in E_k(N,\chi)$ for $i=1,\dots,r$. It follows that $\chi^\sigma=\chi$ for all $\sigma\in G$. We therefore conclude that $\chi$ is real-valued. Because the coefficients of $E_i$ are not all rational numbers, there exists $\sigma\in G$ such that $E_i^\sigma\neq E_i$. The uniqueness of the representation of $E$ as a linear combination of shifts of newforms shows that there exists $j\in 1,\dots,r$ such that $E_i^\sigma=E_j$ and $\sigma(c_i)=c_j$. It follows that the shift by $B_{d_i}$ of $\Tr(c_i E_i)$ appears in equation (\ref{equation:decomp}). \end{proof}

\subsection{Eisenstein series with non-negative Fourier coefficients}

In light of Lemma \ref{sgnlemma}, it is clear that all of the Fourier coefficients of $E(\chi_0,\chi,k)$ are non-negative rational numbers. Standard estimates for the growth of the Fourier coefficients of Eisenstein series now imply that one may construct an Eisenstein series with non-negative Fourier coefficients by choosing an arbitrary element of $E_k(N,\chi)$ with rational Fourier coefficients and adding to it a sufficiently large multiple of $E(\chi_0,\chi,k)$. The main result of this section is that, up to shifting by the $B_d$ operator, all Eisenstein series in $E_k(N,\chi)$ with non-negative, rational Fourier coefficients arise in this manner. Our proof will make use of the following easy lemma.

\begin{lemma}\label{lemma:negative}
Let $\psi_1,\dots,\psi_m$ be non-principal Dirichlet characters whose moduli all divide a positive integer $N\geq 2$ and $c_1,\dots,c_m$ be nonzero complex numbers so that $f(n):=c_1\psi_1(n)+\cdots + c_m\psi_m(n)$ has its image lying in the rational numbers. If $f$ is not identically zero then there exist infinitely many primes $p$ such that $f(p)<0$.
\end{lemma}
\begin{proof}
By the orthogonality relations we know that $\sum_{n=1}^N c_i\psi_i(n)=0$ for $i=1,\dots,m$, hence $\sum_{n=1}^N f(n)=0$. It follows that if $f$ is not identically zero then $f(n)<0$ for some $n\in \{1,\dots,N\}$ and, therefore, for every prime $p\equiv n\pmod {N}$ as well.\end{proof}

\begin{thm}\label{thm:nonnegative}
Let $E\in E_k(N,\chi)$ be a nonzero Eisenstein series with rational Fourier coefficients and suppose that no shift of $E(\chi_0,\chi,k)$ appears in the newform decomposition of $E$. Then $E$ possesses negative Fourier coefficients of arbitrarily large absolute value.
\end{thm}
\begin{proof}
Write $E$ as a linear combination of shifts of newforms of level dividing $N$:

\begin{equation}\label{equation:uniquedecomp}E=\sum_{i=1}^m c_i \left( E_i\mid B_{d_i}\right).\end{equation} 
	
Consider first the special case in which $d_i=1$ for $i=1,\dots,m$. In this case, there exist Dirichlet characters $\chi_1^i,\chi_2^i$ ($i=1,\dots,m$) and complex numbers $c_1,\dots,c_m$ such that:

$$a_E(n)=c_1 \sigma^{k-1}_{\chi^1_1,\chi^1_2}(n)+\cdots + c_m \sigma^{k-1}_{\chi_1^m,\chi_2^m}(n).$$

If $p$ is prime then $\sigma^{k-1}_{\chi_1,\chi_2}(p)=\chi_1(p)+\chi_2(p)p^{k-1}$, hence $a_E(p)=f_1(p)+f_2(p)p^{k-1}$ where $f_1(n)=c_1\chi_1^1(n)+\cdots+c_m\chi_1^m(n)$ and $f_2(n)=c_1\chi_2^1(n)+\cdots+c_m\chi_2^m(n)$. By the linear independence of characters, there exists a prime $q$ such that $f_2(q)\neq 0$. It now follows from Lemma \ref{lemma:negative} that there are infinitely many primes $p$ such that $f_2(p)<0$, and consequently that $a_E(p)<0$ for infinitely many primes. The relation $a_E(p)=f_1(p)+f_2(p)p^{k-1}$ implies that $|a_E(p)|\rightarrow \infty$ as $p\rightarrow \infty$, finishing our proof in this case.

We now consider the general case. Let $d=\min\{d_i: 1\leq i \leq m\}.$ Then there exists an Eisenstein series $E^\prime$ such that $E^\prime\mid B_d$ appears in (\ref{equation:uniquedecomp}) and is maximal in the sense that the newform decomposition of $E-E^\prime\mid B_d$ contains no shifts by the $B_d$ operator. By the previous paragraph, there exist infinitely many primes $p$ such that $a_{E^\prime}(p)<0$. By equation (\ref{equation:uniquedecomp}) and the definition of $E^\prime$, for all but finitely many primes we have $a_E(dp)=a_{E^\prime}(p)$, finishing the proof.\end{proof}

%%%Acknowledgements

{\bf Acknowledgements:} We would like to thank Steven J. Miller for posing the question answered by Theorem \ref{thm:quadtwists}. We would also like to thank Paul Pollack and Daniel Fiorilli for useful discussions related to their work on the least non-residue problem and prime number races, respectively. We are grateful to Tom Shemanske and Paul Pollack for their comments on an earlier draft of this paper. Finally, we would like to express our gratitude to Micah Milinovich and Dimitris Koukoulopoulos for pointing out some relevant results from the literature.

\providecommand{\bysame}{\leavevmode\hbox
to3em{\hrulefill}\thinspace}
\providecommand{\MR}{\relax\ifhmode\unskip\space\fi MR }
\providecommand{\nMRhref}[2]{%
  \href{http://www.ams.org/mathscinet-getitem?mr=#1}{#2}
} \providecommand{\href}[2]{#2}

\end{document}